\documentclass{amsart}
\usepackage{mathrsfs,amssymb,multirow,setspace}

\theoremstyle{plain}
\newtheorem{theorem}{Theorem}[section]
\newtheorem{lemma}[theorem]{Lemma}
\newtheorem{proposition}[theorem]{Proposition}
\newtheorem{corollary}[theorem]{Corollary}

\theoremstyle{definition}
\newtheorem{notation}[theorem]{Notation}
\newtheorem{example}[theorem]{Example}
\newtheorem{definition}[theorem]{Definition}

\theoremstyle{remark}
\newtheorem{remark}[theorem]{Remark}

\input xy
\xyoption{all}

\def\g{\gamma}
\def\G{\Gamma}
\def\vt{\vartheta}
\def\vn{\varnothing}
\newcommand{\sis}[2]{^{#1}\!\zeta_{#2}}
\begin{document}


\title[Rank Properties of $A^+(B_n)$]{The Ranks of the Additive Semigroup Reduct of Affine Near-Semiring over Brandt Semigroup}
\author[Jitender Kumar, K. V. Krishna]{Jitender Kumar and  K. V. Krishna}
\address{Department of Mathematics, Indian Institute of Technology Guwahati, Guwahati, India}
\email{\{jitender, kvk\}@iitg.ac.in}


\begin{abstract}
This work investigates the rank properties of $A^+(B_n)$, the additive semigroup reduct of affine near-semiring over Brandt semigroup $B_n$. In this connection, this work reports the ranks $r_1$, $r_2$, $r_3$ and $r_5$ of $A^+(B_n)$ and identifies a lower bound for the upper rank $r_4(A^+(B_n))$. While this lower bound is found to be the $r_4(A^+(B_n))$ for $n \ge 6$, in other cases where $2 \le n \le 5$, the upper rank of $A^+(B_n)$ is still open for investigation.
\end{abstract}

\subjclass[]{20M10}

\keywords{Rank properties, Semigroup, Near-semiring, Brandt semigroup, Affine maps}

\maketitle


\section*{Introduction}

Since the work of Marczewski in \cite{a.mar66}, many authors have studied the rank properties in the context of general algebras. The concept of rank for general algebras is equivalent to the concept of dimension in linear algebra. The dimension of a vector space is the maximum cardinality of an independent subset, or equivalently, it is the minimum cardinality of a generating set of the vector space.
A subset $U$ of a semigroup $\G$ is said to be \emph{independent} if every element of $U$ is not in the subsemigroup generated by the remaining elements of $U$, i.e. \[ \forall a \in U, \; a \notin \langle U \setminus \{a\} \rangle .\] This definition of independence is equivalent to the usual definition of independence in linear algebra. It can be observed that the minimum size of a generating set need not be equal to the maximum size of an  independent set in a semigroup. Accordingly, Howie and Ribeiro have considered the following possible definitions of ranks for a semigroup $\G$ (cf. \cite{a.hw99,a.hw00}).
\begin{enumerate}
\item $r_1(\G) = \max\{k : \forall U \subseteq \G$ with $|U| = k, U$ is independent\}.
\item $r_2(\G) = \min\{|U| : U \subseteq \G, \langle U\rangle = \G\}$.
\item $r_3(\G) = \max\{|U| : U \subseteq \G, \langle U\rangle = \G, U$ is independent\}.
\item $r_4(\G) = \max\{|U| : U \subseteq \G, U$  is independent\}.
\item $r_5(\G) = \min\{k : \forall U \subseteq \G$ with $|U| = k, \langle U\rangle = \G\}$.
\end{enumerate}
For a finite semigroup $\G$, it can be observed that \[r_1(\G) \le r_2(\G) \le r_3(\G) \le r_4(\G) \le r_5(\G).\] Thus,
$r_1(\G), r_2(\G), r_3(\G), r_4(\G)$ and $r_5(\G)$ are, respectively, known as \emph{small rank}, \emph{lower rank}, \emph{intermediate rank}, \emph{upper rank} and \emph{large rank} of $\G$.

While all these five ranks coincide for certain semigroups, there exist semigroups for which all these ranks are distinct. For instance, all these five ranks are equal to $|\G|$ for a finite left or right zero semigroup $\G$.  For $n > 2$,
Howie et al. have determined all these five ranks for Brandt semigroup $B_n$ (see Definition \ref{d.bs}) through the papers \cite{a.howie87,a.hw99,a.hw00} and observed that all these five ranks are different from each other.

The ranks of rectangular bands and monogenic semigroups were also established in \cite{a.hw99,a.hw00}. The lower rank of completely 0-simple semigroups was obtained by Ru\v{s}kuc \cite{a.ruskuc94}. The intermediate rank of $S_n$, the symmetric group of degree $n$, is determined to be $n-1$ by Whiston \cite{a.whiston00}. All independent generating sets of size $n-1$ in $S_n$ were investigated in \cite{a.cam02}. In \cite{jdm02}, Mitchell studied the rank properties of various groups, semigroups and semilattices. The rank properties of certain semigroups of order preserving transformations have been investigated in \cite{a.howie92} and further extended to orientation-preserving transformations in \cite{a.ping11}.

In this work, we investigate all the five ranks of the additive semigroup reduct of $A^+(B_n)$ -- the affine near-semiring over Brandt semigroup $B_n$. This semigroup is indeed the semigroup generated by affine maps over $B_n$.  The remaining paper has been organized into six sections. Section 1 provides a necessary background material for the subsequent four main sections which are devoted for all the five ranks of $A^+(B_n)$. We conclude the paper in Section 6.

\section{Preliminaries}

In this section, we provide a necessary background material and fix our
notation. For more details one may refer to \cite{a.jk13}.

\begin{definition}
An algebraic structure $(S, +, \cdot)$ is said to be a
\emph{near-semiring} if
\begin{enumerate}
\item $(S, +)$ is a semigroup,
\item $(S, \cdot)$ is a semigroup, and
\item $a(b + c) = ab + ac$, for all $a,b,c \in S$.
\end{enumerate}
\end{definition}

In this work, unless it is required, algebraic structures (such as
semigroups, groups, near-semirings) will simply be referred by their
underlying sets without explicit mention of their operations. Further, we
write an argument of a function on its left, e.g. $xf$ is the value of a
function $f$ at an argument $x$.

\begin{example}
Let $(\Gamma, +)$ be a semigroup and $M(\Gamma)$ be the set of all
mappings on $\Gamma$. The algebraic structure $(M(\G), +, \circ)$ is a
near-semiring, where $+$ is point-wise addition and $\circ$ is composition
of mappings, i.e., for $\gamma \in \Gamma$ and $f,g \in M(\Gamma)$,
$$\g(f + g)= \g f + \g g \;\;\;\; \text{and}\;\;\;\; \g(f \circ g) = (\g
f)g.$$ Also, certain subsets of $M(\G)$ are near-semirings. For instance,
the set $M_c(\Gamma)$ of all constant mappings on $\Gamma$ is a
near-semiring with respect to the above operations so that $M_c(\Gamma)$
is a subnear-semiring of $M(\G)$.
\end{example}

Now, we recall the notion of affine near-semirings from
\cite{kvk05a}. Let $(\G, +)$ be a semigroup. An element $f \in M(\G)$ is
said to be an \emph{affine map} if $f = g + h$, for some $g \in
End(\G)$, the set of all endomorphisms over $\G$, and $h \in M_c(\G)$. This sum is said to be an \emph{affine decomposition} of $f$.
The set of all affine mappings over $\G$, denoted by $\text{Aff}(\G)$, need not be a
subnear-semiring of $M(\G)$. The \emph{affine near-semiring}, denoted by
$A^+(\G)$, is the subnear-semiring generated by $\text{Aff}(\G)$ in
$M(\G)$. Indeed, the subsemigroup of $(M(\G), +)$ generated by
$\text{Aff}(\G)$ equals $(A^+(\G), +)$ (cf. \cite[Corollary 1]{kvk05b}).
If $(\G, +)$ is commutative, then $\text{Aff}(\G)$ is a subnear-semiring
of $M(\G)$ so that $\text{Aff}(\G) = A^+(\G)$.

\begin{definition}\label{d.bs}
For any integer $n \geq 1$, let $[n] = \{1,2,\ldots,n\}$. The semigroup
$(B_n, +)$, where $B_n = ([n]\times[n])\cup \{\vartheta\}$ and the
operation $+$ is given by
\[ (i,j) + (k,l) =
                \left\{\begin{array}{cl}
                (i,l) & \text {if $j = k$;}  \\
                \vartheta     & \text {if $j \neq k $}
                  \end{array}\right.  \]
and, for all $\alpha \in B_n$, $\alpha + \vartheta = \vartheta + \alpha =
\vartheta$,
is known as \emph{Brandt semigroup}. Note that $\vartheta$ is the (two
sided) zero element in $B_n$.
\end{definition}

In \cite{a.jk13}, Jitender and Krishna have studied the structure of (both additive and multiplicative) semigroup reducts of
the near-semiring $A^+(B_n)$ via Green's relations.  We now recall the results on $A^+(B_n)$ which are useful in the present work. The following concept plays a vital role in the study of $A^+(B_n)$.

Let $(\Gamma, +)$ be a semigroup with zero element $\vartheta$. For  $f
\in M(\Gamma)$, the \emph{support of $f$}, denoted by supp$(f)$, is
defined by the set
\[ {\rm supp}(f) = \{\alpha \in \Gamma \;|\; \alpha f \neq \vartheta\}.\]
A function $f \in M(\Gamma)$ is said to be of \emph{k-support} if the
cardinality of supp$(f)$ is $k$, i.e. $|{\rm supp}(f)| = k$. If $k =
|\Gamma|$ (or $k = 1$), then $f$ is said to be of \emph{full support} (or
\emph{singleton support}, respectively). For $X \subseteq M(\Gamma)$, we
write $X_k$ to denote the set of all mappings of $k$-support in $X$, i.e.
$$ X_k = \{ f \in X \mid f \; \text{is of $k$-support}\;  \}.$$

\begin{remark}[\cite{a.jk13}]\label{r.ksupport-sum}
For  $f \in M(\G)_k$ and $g \in M(\G)$, we have
$|{\rm supp}(f+g)| \leq k$ and $|{\rm supp}(g + f)| \leq k$.
\end{remark}

For ease of reference,  we continue to use the following notations for the elements of $M(B_n)$, as given in \cite{a.jk13}.

\begin{notation}\

\begin{enumerate}
\item For $c \in B_n$, the constant map that sends all the elements of $B_n$ to $c$ is denoted by $\xi_c$. The set of all constant maps over $B_n$ is denoted by $\mathcal{C}_{B_n}$.
\item For $k, l, p, q \in [n]$, the singleton support map that send $(k, l)$ to $(p, q)$ is denoted by $\sis{(k, l)}{(p, q)}$.
\item For $p, q \in [n]$, the $n$-support map which sends $(i, p)$ (where $1 \le i \le n$) to $(i\sigma, q)$ using a permutation $\sigma \in S_n$ is denoted by $(p, q; \sigma)$.
\end{enumerate}
\end{notation}

Note that $A^+(B_1) = \{(1, 1; id)\} \cup \mathcal{C}_{B_n}$, where $id$ is the
identity permutation on $[n]$. For $n \ge 2$, the elements of $A^+(B_n)$ are given by the following theorem.

\begin{theorem}[\cite{a.jk13}]\label{t.class.a+bn}
For $n \geq 2$, $A^+(B_n)$ precisely contains $(n! + 1)n^2 + n^4 + 1$ elements with the following breakup.
\begin{enumerate}
\item All the $n^2 + 1$ constant maps.
\item All the $n^4$ singleton support maps.
\item The remaining $(n!)n^2$ elements are the $n$-support maps of the form $(p, q; \sigma)$, where $p, q \in [n]$ and $\sigma \in S_n$.
\end{enumerate}
\end{theorem}

As shown in the Remark \ref{r.ad.a+bn}, except singleton support maps, all other elements of $A^+(B_n)$ are indeed affine maps over $B_n$. We require the following proposition.

\begin{proposition}[\cite{a.jk13}]\label{sn-iso-autbn}
The assignment $\sigma \mapsto \phi_\sigma: S_n \rightarrow Aut(B_n)$ is an isomorphism, where the mapping $\phi_\sigma: B_n \rightarrow B_n$ is given by, $\forall i, j \in [n]$, \[(i, j)\phi_\sigma = (i\sigma, j\sigma)\; \mbox{ and }\; \vt\phi_\sigma = \vt.\]
\end{proposition}

\begin{remark}[\cite{a.jk13}]\label{r.ad.a+bn}
For $k, l, p, q \in [n]$ and $\sigma \in S_n$, we have the following.
\begin{enumerate}
\item Since $\xi_{(p, p)} + \xi_{(p, q)}$ and $\xi_\vt + \xi_{(1,1)}$, respectively,  are affine decompositions of $\xi_{(p, q)}$ and $\xi_{\vt}$, all constant maps are affine maps.
\item The $n$-support map $(p, q; \sigma)$ is an affine map. An affine decomposition for $(p, q; \sigma)$ is $\phi_\sigma + \xi_{(p\sigma, q)}$, where $\phi_\sigma$ is as per Proposition \ref{sn-iso-autbn}. Furthermore, for $f \in Aut(B_n)$, $f + \xi_{(r, s)}$ is an $n$-support map represented by $(r\rho^{-1}, s; \rho)$, where $f = \phi_\rho$.
\item Every singleton support map can be written as sum of a constant map and an $n$-support affine map so that it is an element of $A^+(B_n)$. For instance, $\sis{(k, l)}{(p, q)}$ $= \xi_{(p, q)} + g$, where $g = (l, q; \rho)$ such that $k\rho = q$.
\end{enumerate}
\end{remark}

In what follows, $A^+(B_n)$ denotes the additive  semigroup reduct $(A^+(B_n), +)$ of the affine near-semiring $(A^+(B_n), +, \circ)$. We now present a necessary result on the Green's relations $\mathcal{R}$ and $\mathcal{L}$ of the additive semigroup $A^+(B_n)$.

\begin{theorem}[\cite{a.jk13}]\label{t.gr-rl}
For $1 \le i \le 2$, let
$\pi_i : [n] \times [n] \rightarrow [n]$ be the $i$th projection map. That
is, $(p, q)\pi_1 = p$ and $(p, q)\pi_2 = q$, for all $(p, q) \in [n]
\times [n]$.
\begin{enumerate}
\item For $f,g \in A^+(B_n) \setminus \{\xi_\vt\}$,  $f \mathcal{R} g$ if
and only if ${\rm supp}(f) = {\rm supp}(g)$ and $\alpha f\pi_1 = \alpha
g\pi_1$, for all $\alpha \in {\rm supp}(f)$.

\item For $f,g \in \mathcal{C}_{B_n} \setminus \{\xi_\vt\}$,
$f \mathcal{L} g$ if and only if $\alpha f\pi_2 = \alpha
g\pi_2$, for all $\alpha \in B_n$.

\item The number of $\mathcal{R}$-classes containing $n$-support elements in $A^+(B_n)$ is $(n!)n$.
\end{enumerate}
\end{theorem}

\section{The ranks $r_1$ and $r_2$ of $A^+(B_n)$}

It can be easily observed that $A^+(B_1)$ is an independent set and none of its proper subsets generates $A^+(B_1)$. Hence, for $1 \le i \le 5$, we have \[r_i(A^+(B_1))  = |A^+(B_1)| = 3.\] In the rest of the paper we shall investigate the ranks of $A^+(B_n)$, for $n > 1$.

In this section, after quickly ascertaining the small rank $r_1$ of $A^+(B_n)$, we will obtain its lower rank $r_2$. The small rank of $A^+(B_n)$ comes as a consequence of the following result due to Howie and Ribeiro.

\begin{theorem}[\cite{a.hw00}]\label{r1-gm}
Let $\G$ be a finite semigroup, with $|\G| \ge 2$. If $\G$ is not a band, then $r_1(\G) = 1$.
\end{theorem}

Owing to the fact that $A^+(B_n)$ (for $n \ge 2$) have some non idempotent elements, it is not a band. For instance, the constant maps $\xi_{(p, q)}$ with $p \ne q$ in $A^+(B_n)$  are not idempotent. Hence, we have the following corollary of Theorem \ref{r1-gm}.

\begin{corollary}
For $n \ge 2$, $r_1(A^+(B_n)) = 1$.
\end{corollary}

Now, in the remaining section, we construct a generating set of the minimum cardinality of $A^+(B_n)$ and obtain its lower rank in Theorem \ref{r2-a+bn}.  Consider the subsets \[\mathcal{S} = \{ \xi_{(i, i + 1)} \mid i \in [n - 1] \} \cup \{ \xi_{(n, 1)} \}\] and
\[\mathcal{T} = \{g + h \mid g \in Aut(B_n), \; h \in \mathcal{S} \} \] of $A^+(B_n)$. We develop a proof of Theorem \ref{r2-a+bn} through a sequence of lemmas by showing that the set $\mathcal{S} \cup \mathcal{T}$ serves our purpose.

\begin{lemma}\label{gen-cm-a+bn}
For $n \ge 2$, $\langle\mathcal{S}\rangle = \mathcal{C}_{B_n}$.
\end{lemma}

\begin{proof}
Let $f \in \mathcal{C}_{B_n}$; then, either $f = \xi_\vt$ or $f = \xi_{(i, j)}$. If $f  = \xi_\vt$, then, for $p \in [n-1]$, write   $\xi_\vt = \xi_{(p, p + 1)} + \xi_{(p, p + 1)}$ so that $f \in \langle \mathcal{S}  \rangle$. If $f = \xi_{(i, j)}$, then, for $i < j$, we have \[\xi_{(i, j)} = \xi_{(i, i + 1)} + \xi_{(i + 1, i + 2)} + \cdots + \xi_{(j - 1, j)},\] and, for $i \ge j$, \[\xi_{(i, j)} = \xi_{(i, i + 1)} + \xi_{(i + 1, i + 2)} + \cdots + \xi_{(n - 1, n)} + \xi_{(n, 1)} + \xi_{(1, 2)}+  \cdots +\xi_{(j-1, j)}\] so that $f \in \langle \mathcal{S}\rangle$.
\end{proof}

\begin{lemma}\label{gen-a+bn}
If $X \subset A^+(B_n)$ such that $\langle X \rangle = \mathcal{C}_{B_n}$, then $\langle X \cup \mathcal{T} \rangle = A^+(B_n)$.
\end{lemma}

\begin{proof}
Since the constant elements of $A^+(B_n)$ are generated by $X$, in view of Theorem \ref{t.class.a+bn}, it is sufficient to prove that the $n$-support and singleton support elements are generated by $X \cup \mathcal{T}$. However, since every singleton support map is a sum of a constant map and an $n$-support map (cf. Remark \ref{r.ad.a+bn}(3)), we will now observe that $X \cup \mathcal{T}$ generates the $n$-support maps of $A^+(B_n)$.

Let $f \in A^+(B_n)_n$. By Remark \ref{r.ad.a+bn}(2), $f = g + \xi_c$ for some $g \in Aut(B_n)$ and $c \in B_n \setminus \{\vt\}$. By Lemma \ref{gen-cm-a+bn}, write $\xi_c = \displaystyle \sum_{i = 1}^{k}f_i$, for some $f_i$'s from $\mathcal{S}$ so that $$f = g + \displaystyle \sum_{i = 1}^{k}f_i = g + f_1 + \displaystyle \sum_{i = 2}^{k}f_i.$$ Note that $g + f_1 \in \mathcal{T}$ and each $f_i$ is a sum of elements of $X$. Hence, $f \in  \langle X \cup \mathcal{T}  \rangle$.
\end{proof}

\begin{lemma}\label{gen-nsm-a+bn}
For $g_1, g_2, g_3 \in Aut(B_n)$ with $g_2 \ne g_3$ and $p, q, s, t \in [n]$ such that $p \ne s$, we have the following.
\begin{enumerate}
\item If $f_1 = g_1 + \xi_{(p, q)}$ and  $f_1' = g_1 + \xi_{(s, t)}$, then $f_1 \ne f_1'$.
\item If $f_2 = g_2 + \xi_{(p, q)}$ and  $f_3 = g_3 + \xi_{(p, q)}$, then $f_2 \ne f_3$.
\end{enumerate}
Hence, $|\mathcal{T}| = n(n!)$.
\end{lemma}

\begin{proof}
As per Proposition \ref{sn-iso-autbn}, for $1 \le i \le 3$, let $g_i = \phi_{\sigma_i}$ so that $f_1, f_1', f_2, f_3$ are $n$-support maps represented by $(p \sigma_1^{-1}, q; \sigma_1)$,  $(s \sigma_1^{-1}, t; \sigma_1)$, $(p \sigma_2^{-1}, q; \sigma_2)$ and $(p \sigma_3^{-1}, q; \sigma_3)$,  respectively (cf. Remark \ref{r.ad.a+bn}(2)).

(1) Since $p \ne s$ and $\sigma_1$ is a permutation on $[n]$, $f_1$ and $f_1'$ have different support so that $f_1 \ne f_1'$.

(2) If $p \sigma_2^{-1} \ne p \sigma_3^{-1}$, then we are done. Otherwise, since $\sigma_2 \ne \sigma_3$, there exists $i_0 \in [n]$ such that $i_0 \sigma_2 \ne i_0 \sigma_3$. Now, \[(i_0, p \sigma_2^{-1})f_2 = (i_0 \sigma_2, q) \ne (i_0 \sigma_3, q) = (i_0, p \sigma_3^{-1})f_3 \]so that $f_2 \ne f_3$.

Now, since $|Aut(B_n)| = n!$ (cf. Proposition \ref{sn-iso-autbn}) and there are $n$ elements in $\mathcal{S}$, we have
$|\mathcal{T}| = n(n!)$.

\end{proof}

\begin{lemma}\label{l.fs-sum}
For $1 \le i \le k$, let $f, f_i \in A^+(B_n)$ such that $f = \displaystyle{\sum_{i = 1}^{k}}f_i$.
\begin{enumerate}
\item If $f \in A^+(B_n)_{n^2 + 1}$, then $f_1 \in R_f$ and $f_k \in L_f$.
\item If $f \in A^+(B_n)_{n}$, then $f_1 \in R_f$.
\end{enumerate}
\end{lemma}

\begin{proof}
(1) If $f \in A^+(B_n)_{n^2 + 1}$, then $f_i \in A^+(B_n)_{n^2 + 1}$, for all $i$ (cf.  Remark \ref{r.ksupport-sum}). Let $f = \xi_{(p, q)}$ and $f_i = \xi_{(p_i, q_i)}$ so that \[\xi_{(p, q)} = \displaystyle{\sum_{i = 1}^{k}}\xi_{(p_i, q_i)}.\] Then clearly, for $1 \le i \le k - 1$,  $q_i = p_{i + 1}$  and $p_1 = p$, $q_k = q$. Hence, by Theorem \ref{t.gr-rl}, $f_1 \in R_f$ and $f_k \in L_f$.\\

(2) If $f \in A^+(B_n)_{n}$, then $|{\rm supp}(f_i)| \ge n$, for all $i$ (cf. Remark \ref{r.ksupport-sum}). Then, for each $i$, $|{\rm supp}(f_i)| = n$ or $n^2 + 1$ (cf. Theorem \ref{t.class.a+bn}). Note that, there exists $j$ ($1 \le j \le k$) such that $|{\rm supp}(f_j)| = n$; otherwise, $f$ will be a constant map. If $j \ge 2$, then, by \cite[Proposition 2.9]{a.jk13}, $|{\rm supp}(f_{j-1} + f_{j})| \le 1$ so that $|{\rm supp}(f)| \le 1$; a contradiction. Thus, we have  $j = 1$ and, for all $i > 1$, $|{\rm supp}(f_i)| = n^2 + 1$.

Let $(k, q; \sigma)$  and $(k', p; \tau)$ be the representations of $f$ and $f_1$, respectively,  and $\displaystyle \sum_{i = 2}^{k} f_i = \xi_{(r, s)}$ for some $r, s \in [n]$. Then, note that $p = r$, $k' = k$, $\tau = \sigma$ and $s = q$. Thus, $f_1 = (k, r; \sigma)$ and $\displaystyle \sum_{i = 2}^{k} f_i = \xi_{(r, q)}$ for some $r \in [n]$. Hence, by Theorem \ref{t.gr-rl}, $f_1 \in R_f$.
\end{proof}

\begin{lemma}\label{gen-set-prop}
Every generating set of $A^+(B_n)$ contains at least
\begin{enumerate}
\item $n$ elements of full support, and
\item $n(n!)$ elements of $n$-support.
\end{enumerate}
\end{lemma}

\begin{proof}
Let $V$ be a generating set of $A^+(B_n)$. For $f \in A^+(B_n)$, write $f = \displaystyle \sum_{i = 1}^{k}f_i$, for some $f_i \in V$.

(1) If $f \in \mathcal{S}$, then by Lemma \ref{l.fs-sum}(1), $f_1 \in R_f$ so that $R_f \cap V  \ne \vn$. Further, one can observe that if $g, h \in \mathcal{S}$ with $g \ne h$, then $R_g \cap R_h = \vn$ (cf. Theorem \ref{t.gr-rl}(1)). Hence, since $\mathcal{S}$ has $n$ full support elements, $V$ will have at least $n$ full support elements.

(2) If $f \in \mathcal{T}$, then by Lemma \ref{l.fs-sum}(2), we have $R_f \cap V  \ne \vn$.  For $g, h \in \mathcal{T}$ with $g \ne h$, we show that  $R_g \cap R_h = \vn$ so that $|V| \ge |\mathcal{T}| = (n!)n$. In view of Remark \ref{r.ad.a+bn}(2), by considering $n \mod n = n$, write $g = (r\sigma^{-1}, r + 1 \mod n; \sigma)$ and $h = (s\rho^{-1}, s + 1 \mod n; \rho)$. Since $g \ne h$, either $\sigma \ne \rho$ or $r \ne s$. Hence, by Theorem \ref{t.gr-rl}(1), $R_g \ne R_h$ in either case.
\end{proof}

In view of lemmas \ref{gen-cm-a+bn} and \ref{gen-a+bn}, we have the following corollary of Lemma \ref{gen-set-prop}.

\begin{corollary}\label{sut-min-car}
$|\mathcal{S} \cup \mathcal{T}|$ is the minimum such that $\langle \mathcal{S} \cup \mathcal{T} \rangle = A^+(B_n)$.
\end{corollary}

Combining the results from Lemma \ref{gen-cm-a+bn} through Corollary \ref{sut-min-car}, we have the following main theorem of the section.

\begin{theorem}\label{r2-a+bn}
For $n \ge 2$, $r_2(A^+(B_n)) = n(n! + 1)$.
\end{theorem}

\section{Intermediate rank}

In this section, after ascertaining certain relevant properties of independent generating sets of $A^+(B_n)$, we obtain its intermediate rank.

\begin{lemma}\label{max-igs}
Let $U$ be an independent generating set of $A^+(B_n)$; then
\begin{enumerate}
\item $A^+(B_n)_1 \cap U = \vn$,
\item $|A^+(B_n)_n \cap U| = n(n!)$,
\item $n \le |A^+(B_n)_{n^2 + 1} \cap U| \le 2n - 2$.
\end{enumerate}
Hence, $|U| \le n(n!) + 2n -2$.
\end{lemma}

\begin{proof}$\;$
\begin{enumerate}
\item Let  $f =\; \sis{(k, l)}{(p, q)} \in A^+(B_n)_1 \cap U$. Since $U$ is a generating set, we have $\xi_{(p, r)} \in \langle U \setminus \{f\}\rangle$ and, for $k \sigma = r$,  $(l, q; \sigma) \in \langle U \setminus \{f\} \rangle$ (cf. Remark \ref{r.ksupport-sum}). Note that $f = \xi_{(p, r)} + (l, q; \sigma)$ so that $f \in \langle U \setminus \{f\} \rangle$; a contradiction to $U$ is an independent set.

\item By Lemma \ref{gen-set-prop}(2), $|A^+(B_n)_n \cap U| \ge n(n!)$. Since $A^+(B_n)_n$ contains only $n(n!)$ $\mathcal{R}$-classes (cf. Theorem \ref{t.gr-rl}(3)), if $U$ contain more than $n(n!)$ elements of $n$-support, then there exist distinct $f,g \in A^+(B_n)_n \cap U$ such that $f \mathcal{R} g$. By Theorem \ref{t.gr-rl}, $f = (k, p; \sigma)$ and $g = (k, p'; \sigma)$, for some $\sigma \in S_n$. Note that, $g = f + \xi_{(p, p')}$, where $\xi_{(p, p')} \in \langle  U \setminus \{g\}\rangle$, so that $ g \in \langle  U \setminus \{g\} \rangle$; a contradiction to independence of $U$. Hence, there are exactly $n(n!)$ elements of $n$-support  in $U$.

\item By Lemma \ref{gen-set-prop}(1), $|A^+(B_n)_{n^2 + 1} \cap U| \ge n$. Further, using \cite[Theorem 3.1]{a.hw99}, one can observe that $U$ contains at most $2n-2$ full support maps.
\end{enumerate}
\end{proof}

\begin{theorem}
For $n \ge 2$, $r_3(A^+(B_n)) = n(n!) + 2n - 2$.
\end{theorem}

\begin{proof}
First note that the set \[\mathcal{S'} = \{\xi_{(1, i)} \mid 2 \le i \le n \} \cup \{\xi_{(j, 1)} \mid 2 \le j \le n \}\]  generates all constant maps in $A^+(B_n)$. For instance, $\xi_\vt = \xi_{(1, 2)} + \xi_{(1, 2)}$ and $\xi_{(1, 1)} = \xi_{(1, 2)} + \xi_{(2, 1)}$. Now, for $p, q \in [n]$, clearly $\xi_{(1, p)}$ and $\xi_{(q, 1)} \in \langle \mathcal{S'} \rangle$. Further, since $\xi_{(p, q)} = \xi_{(p, 1)} + \xi_{(1, q)}$, we have $\xi_{(p, q)} \in \langle \mathcal{S'} \rangle$.

Hence, by Lemma \ref{gen-a+bn}, the set $V = \mathcal{S'} \cup \mathcal{T}$ is a generating set of $A^+(B_n)$. We show that $V$ is also an independent set. Since $|V| = n(n!) + 2n - 2$, by Lemma \ref{max-igs}, the theorem follows.

\textit{$V$ is an independent set}: For $f \in V$, suppose $f = \displaystyle\sum_{i =1}^k f_i$ with $f_i \in V \setminus \{f\}$. Then, by Lemma \ref{l.fs-sum}, $f_1 \in R_f$. If $f \in \mathcal{S'} \cup \mathcal{T}$, in the following, we observe that $f = f_1$; which is a contradiction so that $f \notin \langle V \setminus \{f\} \rangle$.

If $f = \xi_{(q, 1)} \in \mathcal{S'}$ (for some $2 \le q \le n$), then $f_1 = \xi_{(q, l)}$, for some $l \in [n]$ (cf. Theorem \ref{t.gr-rl}(1)).  Hence, $f_1 =  f$. For $2 \le p \le n$,  if $f = \xi_{(1, p)}$, the argument is similar.

If $f = (k, p; \sigma) \in \mathcal{T}$ (for some $k, p \in [n], \sigma \in S_n$), again by Theorem \ref{t.gr-rl}(1), $f_1 = (k, s; \sigma)$ for some $s \in [n]$.  Consequently,  $f = f_1$ (cf. construction of $\mathcal{T}$).
\end{proof}

\section{Upper rank}

It is always difficult to identify the upper rank of a semigroup and we observe that $A^+(B_n)$ is also not an exception. In order to investigate the upper rank $r_4(A^+(B_n))$, in this section, first we obtain a lower bound for the upper rank and eventually we prove that this lower bound is indeed the $r_4(A^+(B_n))$, for $n \ge 6$. We also report an independent set of 14 elements in $A^+(B_2)$.

\begin{theorem}\label{lb-r4a+bn}
For $n \ge 2$, $ I = A^+(B_n)_n \cup \{\xi_{(i, i)} : i \in [n]\}$ is an independent set in $A^+(B_n)$. Hence, by Theorem \ref{t.class.a+bn}, $r_4(A^+(B_n)) \ge (n!)n^2 + n$.
\end{theorem}

\begin{proof}
For $f \in  I$, suppose  $f = \displaystyle{\sum_{j = 1}^{k}f_j}$, for $f_j \in  I$. We prove that $f_1 = f$ so that $f \notin \langle  I \setminus \{f\}\rangle$.
Let $f = \xi_{(i, i)}$;  then clearly $f_j = f$ for all $j$. We may now suppose $f \in A^+(B_n)_n$ and $(k, p; \sigma)$ be the representation of $f$. By Lemma \ref{l.fs-sum}(2), $f_1 \in R_f$ and $\displaystyle{\sum_{i = 2}^{k}f_i} = \xi_{(s, p)}$ for some $s \in [n]$. Note that $f_1 = (k, s; \sigma)$  (cf. Theorem \ref{t.gr-rl}).
If $s \ne p$, then $\xi_{(s, p)} \notin \langle I\rangle$; a contradiction. Hence, $s = p$ so that $f_1 = f$.
\end{proof}

\begin{corollary}\label{r.mi-ns}
$A^+(B_n)_n$ is an independent subset of size $(n!)n^2$ in $A^+(B_n)$.
\end{corollary}

\begin{lemma}\label{l.ind-ss}
Let $Q$ be an independent subset of $B_n$ and
\[Q' = \left.\left \{ \sis{(k, l)}{\alpha} \right| k, l \in [n] \mbox{ and
} \alpha \in Q \right \};\] then $Q'$ is an independent subset of $A^+(B_n)$.
\end{lemma}

\begin{proof}
For $\sis{(k, l)}{\alpha} \in Q'$, $k_j, l_j \in [n]$ and $\alpha_j \in Q$ suppose  \[\sis{(k, l)}{\alpha}  = \sum_{j = 1}^{k}\; \sis{(k_j, l_j)}{\alpha_j}.\] Clearly $k_j = k$, $l_j = l$ for all $j$, and $\alpha = \displaystyle\sum_{j=1}^k\alpha_j$. Since $Q$ is independent, we have $\alpha = \alpha_i$ for some $i$ ($1 \le i \le k$). Consequently,  $\sis{(k, l)}{\alpha} \notin \langle Q'
\setminus \{\sis{(k, l)}{\alpha}\}\rangle$ so that $Q'$ is an independent set.
\end{proof}

\begin{remark}\label{r.mi-fs}
Since $B_n$ is isomorphic to the semigroup $\mathcal{C}_{B_n}$, by \cite[Theorem 3.3]{a.hw99}, we have
$r_4(\mathcal{C}_{B_n}) = \left\lfloor n^2 /4 \right\rfloor \ + n$.
\end{remark}

In view of Remark \ref{r.mi-fs}, we have the following corollary of Lemma \ref{l.ind-ss}.

\begin{corollary}\label{r.mi-ss}
For $n \ge 2$,  the maximum size of an independent subset in $A^+(B_n)_1$ is $n^2(\left\lfloor n^2 /4 \right\rfloor \ + n )$.
\end{corollary}

\begin{remark}\label{r.ins-fs}
For $f_i \in A^+(B_n)$, if $\displaystyle\sum_{i=1}^rf_i + \xi_{(p, p)} + \sum_{i = r+1}^s f_i$ is nonzero, then the sum equals $\displaystyle\sum_{i=1}^sf_i$.
\end{remark}

For $n =2$, we provide a better lower bound in the following theorem.

\begin{theorem}\label{lb-r4a+b2}
$r_4(A^+(B_2)) \ge 14$.
\end{theorem}

\begin{proof}

We claim that the 14-element set \[P = \left.\left \{ \sis{(k, l)}{\alpha} \right| k, l \in [2] \mbox{ and } \alpha \in Q \right \} \cup \{\xi_{(1, 1)}, \xi_{(2, 2)}\},\] where $Q = \{(1, 1), (1, 2), (2, 2)\}$, is an independent subset of the semigroup $A^+(B_2)$.

For $f \in P$, suppose  $f = \displaystyle{\sum_{j = 1}^{k}f_j}$, for $f_j \in P$. If $f = \xi_{(i, i)}$, then $f_j = f$ for all $j$ so that $f \notin \langle P \setminus \{f\}\rangle$. Otherwise, $f$ = $\sis{(k, l)}{\alpha}$ for $\alpha \in Q$. By Remark \ref{r.ins-fs}, the sum for $f$ can be reduced to a sum with only the singleton support elements of $P$. Hence, from the proof of Lemma \ref{l.ind-ss}, $P$ is independent.
\end{proof}

\begin{theorem}\label{r4-a+bn}
For $n \ge 6$, $r_4(A^+(B_n)) = (n!)n^2 + n$.
\end{theorem}

\begin{proof}
For $n \ge 6$, if an independent subset $K$ of $A^+(B_n)$ contains a single support map or a full support map of the form $\xi_{(p, q)}$, for $p \ne q$, then $|K| < (n!)n^2 + n$. Hence, the result follows by Theorem \ref{lb-r4a+bn}.

Let $K$ be an independent subset of $A^+(B_n)$. By Corollary \ref{r.mi-ns}, Remark \ref{r.mi-fs} and Corollary \ref{r.mi-ss}, we have $|K| \le \kappa$, where \[\kappa = (n!)n^2 + \left\lfloor n^2 /4 \right\rfloor \ + n + n^2(\left\lfloor n^2 /4 \right\rfloor \ + n ).\] In the following, we observe that, out of $\kappa$ (the maximum possible number) elements, at least $(n-1)!(n-1)$ elements will not be in $K$. Hence, since $n \ge 6$, \[|K| \le \kappa - (n-1)!(n-1) <  (n!)n^2 + n.\]

\emph{Case 1: $\xi_{(p, q)} \in K$ with $p \ne q$}. For each $\sigma \in S_n$ and $l \in [n]$, since \[(l, q; \sigma) = (l, p; \sigma) + \xi_{(p, q)},\] the independent set $K$ cannot contain $(l, q; \sigma)$ and $(l, p; \sigma)$ together. Thus, out of $\kappa$ elements, at least $(n!)n$ elements will not be in $K$.

\emph{Case 2: $\sis{(r, s)}{(p, q)} \in K$}. For each $t \in [n]$ and $\sigma, \rho \in S_n$ such that $r \sigma = p$ and $r \rho = t$, since  \[\sis{(r, s)}{(p, q)} = (s, t; \sigma) + (s, q; \rho),\] the independent set $K$ cannot contain $(s, t; \sigma)$ and $(s, q; \rho)$ together.

\begin{description}
\item[{\it Subcase 2.1}] $p = q$. Except at $t = q$, for all other choices, none of the first terms is equal to any of the second terms in the sums $(s, t; \sigma) + (s, q; \rho)$. Thus, out of $\kappa$ elements, at least $(n-1)!(n-1)$ elements (either first terms or second terms in the sums) will not be in $K$.

\item[{\it Subcase 2.2}] $p \ne q$. In the similar lines of \emph{Subcase 2.1}, at least $(n-1)!(n-2)$ elements will not be in $K$ for the choices of $t \in [n] \setminus \{p, q\}$. If $t \in \{p, q\}$, the set of second terms of the sums for $t = p$ is equal to the set of first terms of the sums for $t = q$, which is of size $(n-1)!$. Thus, for $t \in \{p, q\}$, at least $(n-1)!$ elements will not be in $K$. Hence, a total of at least $(n-1)!(n-1)$ elements will not be in $K$.
\end{description}
\end{proof}

\section{Large Rank}

In this section, we obtain the large rank of $A^+(B_n)$.  An element $a$ of a semigroup $(\G, +)$ is said to be \emph{indecomposable} if there do not exist $b , c \in \G \setminus \{a\}$ such that $a = b + c$. The following  key result by Howie and Ribeiro is useful to find the large rank of a finite semigroup.

\begin{theorem}[\cite{a.hw00}]\label{r5-lsgp}
Let $\G$ be a finite semigroup and $V$ be a largest proper subsemigroup of $\G$; then $r_5(\G) = |V|+1.$ Hence, $r_5(\G) = |\G|$ if and only if $\G$ contains an indecomposable element.
\end{theorem}

\begin{remark}
Since $\xi_{(1, 2)}$ is an indecomposable element in $A^+(B_2)$, we have \[r_5(A^+(B_2)) = |A^+(B_2)| = 29.\]
\end{remark}

However, as shown in the following proposition, there is  no indecomposable element in $A^+(B_n)$, for $n \ge 3$.

\begin{proposition}
For $n \ge 3$, all the elements of $A^+(B_n)$ are decomposable.
\end{proposition}

\begin{proof}
Refereing to  Theorem \ref{t.class.a+bn}, we give a decomposition of each element $f \in A^+(B_n)$ in the following cases.
\begin{enumerate}
\item $f$ is the zero element: $\xi_\vt = \xi_{(p, q)} + \xi_{(r, s)}$, for $q \ne r$.
\item $f$ is a full or singleton support element: Let ${\rm Im}(f)\setminus \{\vt\} = \{(p, q)\}$. We have $f =  g + h$, where
$g, h \in A^+(B_n)$ such that ${\rm supp}(f) = {\rm supp}(g) = {\rm supp}(h)$ and ${\rm Im}(g)\setminus \{\vt\} = \{(p, r)\}$, ${\rm Im}(h)\setminus \{\vt\} = \{(r, q)\}$, for some $r \ne p, q$.
\item $f$ is an $n$-support map: Let $f  = (k, p; \sigma)$. Note that $f = (k, q; \sigma) + \xi_{(q, p)}$, for $q \ne p$.
\end{enumerate}

\end{proof}

In order to find the large rank of $A^+(B_n)$, we adopt the technique that is used to find the large rank of Brandt semigroups in \cite{a.jk13-3}. The technique, as stated in Lemma \ref{gm-lsgp}, relies on the concept of prime subsets of semigroups.  A nonempty subset $U$ of a semigroup $(\G, +)$ is said to be \emph{prime} if, $\forall a, b \in \G$, \[a + b \in U \Longrightarrow a \in U \vee b \in U.\]

\begin{lemma}[\cite{a.jk13-3}]\label{gm-lsgp}
Let $V$ be a smallest and proper prime subset of a finite semigroup $\G$; then $\G \setminus V$ is a largest proper subsemigroup of $\G$.
\end{lemma}

Using Lemma \ref{gm-lsgp} and Theorem \ref{r5-lsgp}, now we obtain the large rank of $A^+(B_n)$ in the following theorem.

\begin{theorem}\label{lr-a+bn+}
 For $n \geq 2$, $r_5(A^+(B_n)) = (n!)n^2+n^2+n^4-n+3.$
\end{theorem}

\begin{proof}
We show that the set $V = \{\xi_{(n, k)} \mid 1 \le k \le n - 1\}$ is a smallest prime subset of $A^+(B_n)$. Since, $|V| = n-1$, the result follows from Theorem \ref{t.class.a+bn}.

\emph{$V$ is a prime subset}: For $\xi_{(i, j)}, \xi_{(l, k)} \in A^+(B_n)$, if $\xi_{(i, j)} + \xi_{(l, k)} \in V$, then $i = n$, $j = l$ and $1 \le k \le n-1$. If $l = n$, then clearly $\xi_{(l, k)} \in V$; otherwise, $\xi_{(i, j)} \in V$.

\emph{$V$ is a smallest prime subset}: Let $U$ be a prime subset of $A^+(B_n)$ such that $|U| < |V|$. If $U \subset V$, then let $\xi_{(n, q)} \in V \setminus U$. Now, for $\xi_{(n, p)} \in U$ and for all $i \in [n]$, clearly we have \[\xi_{(n, p)} = \xi_{(n, i)} + \xi_{(i, p)}.\] Note that, for $i = q$, neither $\xi_{(n, i)}$ nor $\xi_{(i, p)}$ is in $U$; a contradiction to $U$ is a prime set.

Otherwise, we have $U \not \subset V$. Let $f \in U \setminus V$; then, $f$ can be (i) $\xi_\vt$, (ii) $\xi_{(n, n)}$, (iii) $\xi_{(p, q)}$, for some $p \in [n-1]$, $q \in [n]$, (iv) an $n$-support map, or (v) a singleton support map. In all the five cases we observe that $|U| \ge n-1$, which is a contradiction to the choice of $U$.

\begin{itemize}
\item[(i)] $f = \xi_\vt$: For each $i \in [n]$, since $\xi_{(i, 1)} + \xi_{(2, i)} = \xi_{\vt}$, there are at least $n$ elements in $U$.

\item[(ii)] $f = \xi_{(n, n)}$: For each $i \in [n-1]$, since $\xi_{(n, i)}+ \xi_{(i, n)} = \xi_{(n, n)}$, there are at least $n-1$ elements in $U$.

\item[(iii)] $f = \xi_{(p, q)}$, for some $p \in [n-1]$, $q \in [n]$: First note that, for each $i \in [n]$, we have $\xi_{(p, i)} + \xi_{(i, q)} = \xi_{(p, q)}$. If $p = q$, the argument is similar to above (ii). Otherwise, corresponding to $n-2$ different choices of $i \ne p, q$, there are at least $n - 2$ elements in $U$. Now, including $f$, we have $|U| \ge n-1$.

\item[(iv)] $f$ is an $n$-support map: Let $f = (k, p; \sigma)$. For $q \in [n] \setminus \{p\}$, since $(k, p; \sigma) = (k, q; \sigma) + \xi_{(q, p)}$, there are at least $n-1$ elements in $U$.

\item[(v)] $f$ is a singleton support map: Let $f$ = $\sis{(k, l)}{(p, q)}$. For $s \in [n] \setminus \{p, q\}$, since $$\sis{(k, l)}{(p, q)} =\; \sis{(k, l)}{(p, s)} +\; \sis{(k, l)}{(s, q)},$$there are at least $n-2$ elements in $U$. Since $f \in U$, we have $|U| \ge n-1$.
\end{itemize}
\end{proof}

\section{Conclusion}

In this work, we have investigated the ranks of $A^+(B_n)$, the additive semigroup reduct of the affine near-semiring over Brandt semigroup. Using the structural properties of $A^+(B_n)$ given in \cite{a.jk13},  we obtained the small, lower, intermediate and large ranks of $A^+(B_n)$, for all $n \ge 1$. The upper rank $r_4(A^+(B_n))$ was found for the semigroups with $n \ge 6$. For $2 \le n \le 5$,  through an explicit construction of an independent set, we reported a lower bound for $r_4(A^+(B_n))$. While 14 is the lower bound for the case $n = 2$, it is $(n!)n^2 + n$ for the other cases. We conjecture that these lower bounds are indeed the upper ranks of the respective cases. In the similar lines of this work, one could also investigate on the rank properties of the multiplicative semigroup reduct of the affine near-semiring over Brandt semigroup.

\end{document}